\documentclass[12pt,a4paper]{amsart}
\usepackage{graphicx,amssymb}
\usepackage{amsmath,amssymb,amscd}
\input xy
\xyoption{all}
\usepackage[all]{xy}
\usepackage{hyperref}

\newcommand{\ra}{\rightarrow}

\newcommand{\ZZ}{\mathbb Z}

\newcommand{\PP}{\mathbb P}

\newcommand{\cO}{\mathcal{O}}

\newcommand{\cM}{\mathcal{M}}

\newcommand{\Pic}{\mbox{Pic}}

\newcommand{\Cl}{\operatorname{Cliff}}

\theoremstyle{plain}
\newtheorem{theorem}{Theorem}[section]
\newtheorem{lem}[theorem]{Lemma}
\newtheorem{prop}[theorem]{Proposition}
\newtheorem{cor}[theorem]{Corollary}

\newtheorem{rem}[theorem]{Remark}

\newtheorem{ques}[theorem]{Question}
\numberwithin{equation}{section}
\begin{document}
\title[Small Clifford index]{Bundles of rank 2 with small Clifford index on algebraic curves}

\author{H. Lange}
\author{P. E. Newstead}

\address{H. Lange\\Department Mathematik\\
              Universit\"at Erlangen-N\"urnberg\\
              Bismarckstra\ss e $1\frac{ 1}{2}$\\
              D-$91054$ Erlangen\\
              Germany}
              \email{lange@mi.uni-erlangen.de}
\address{P.E. Newstead\\Department of Mathematical Sciences\\
              University of Liverpool\\
              Peach Street, Liverpool L69 7ZL, UK}
\email{newstead@liv.ac.uk}

\thanks{Both authors are members of the research group VBAC (Vector Bundles on Algebraic Curves). They would like to thank the Isaac Newton Institute, where a part of this paper was written during the Moduli Spaces programme.}
\keywords{Algebraic curve, stable vector bundle, Clifford index, K3-surface.}
\subjclass[2000]{Primary: 14H60; Secondary: 14J28}

\begin{abstract}
In this paper, we construct stable bundles $E$ of rank 2 on suitably chosen curves of any genus $g\ge12$ with maximal Clifford index such that the Clifford index of $E$ takes the minimum possible value for curves with this property. 
\end{abstract}
\maketitle

\section{Introduction}                                                   

In a previous paper \cite{cl} (see also \cite{fo, fo2, ln}), we constructed examples of curves for which the rank-$2$ Clifford index $\Cl_2(C)$ is strictly less than the classical Clifford index, thus producing counter-examples to a conjecture of Mercat \cite{m}. The purpose of the present paper is to improve \cite[Theorem 1.1]{cl} by substantially weakening the hypotheses; the new result is best possible and enables us to construct examples of curves $C$ of any genus $g\ge12$  for which the Clifford index $\Cl(C)$ takes its maximum possible value $\left[\frac{g-1}2\right]$, while the rank-$2$ Clifford index $\Cl_2(C)$ satisfies $\Cl_2(C)=\frac12\Cl(C)+2$, which is the minimum possible value for curves of Clifford index $\Cl(C)$.

\medskip
To state the results, we recall first the definition of $\Cl_n(C)$. For any vector bundle $E$ of rank $n$ and degree $d$ on $C$, we define
$$
\gamma(E) := \frac{1}{n} \left(d - 2(h^0(E) -n)\right) = \mu(E) -2\frac{h^0(E)}{n} + 2.
$$
If $C$ has genus $g \geq 4$, we then define, for any positive integer $n$,
$$
\Cl_n(C):= \min_{E} \left\{ \gamma(E) \;\left| 
\begin{array}{c} E \;\mbox{semistable of rank}\; n \\
h^0(E) \geq 2n,\; \mu(E) \leq g-1
\end{array} \right. \right\}
$$
(this invariant is denoted in \cite{cliff, cl, ln} by $\gamma_n'$). Note that $\Cl_1(C) = \Cl(C)$ is the usual Clifford index of the curve $C$. Moreover, as observed in \cite[Proposition 3.3 and Conjecture 9.3]{cliff}, the conjecture of \cite{m} can be restated in a slightly weaker form as

\medskip
\noindent{\bf Conjecture.} $\Cl_n(C)=\Cl(C)$. 

\medskip
In fact, for $n=2$, this form of the conjecture is equivalent to the original (see \cite[Proposition 2.7]{ln}).

\medskip
Our main theorem can now be stated.

\medskip
\noindent{\bf  Theorem \ref{thm3.1}.}\begin{em} Suppose that   $g$, $s$, $d$ are integers such that
\begin{equation*}
s \geq -1,\ g \geq 2s+14\mbox{ and } d=g-s.
\end{equation*}
Then there exists a curve $C$ of genus $g$ having $\Cl(C) = \left[ \frac{g-1}{2} \right]$ 
and a stable vector bundle $E$ of rank $2$ and degree $d$ on $C$ with $\gamma(E) = \frac{g-s}{2} - 2$.  Hence
$$
\Cl_2(C) \le \frac{g-s}{2} -2 <\Cl(C).
$$\end{em}

\medskip
This theorem is a substantially strengthened version of \cite[Theorem 1.1]{cl}; the hypotheses are now best possible in the sense that the theorem fails for $g\le2s+13$. The stronger hypotheses in the original theorem were needed to ensure that certain K3-surfaces contained no effective divisors $D$ such that $D^2=0$ or $D^2=-2$. In the present paper, our K3-surfaces may contain such divisors, but we are able to control these and show that they do not affect the calculations required to prove the theorem. The proof of the theorem itself is essentially the same as that of \cite[Theorem 1.1]{cl}; we give it in full for the sake of clarity and to demonstrate how the hypotheses are used.

\medskip
As a corollary to Theorem \ref{thm3.1} we have

\medskip
\noindent{\bf Theorem \ref{thm2}.}\begin{em} Let $\gamma$ be an integer, $\gamma\ge5$. Then there exists a curve $C$ with $\Cl(C)=\gamma$ such that
$$\Cl_2(C)=\frac{\gamma}2+2.$$
Moreover $C$ can be taken to have genus either $2\gamma+1$ or $2\gamma+2$.\end{em}

\medskip
Following an extended discussion of curves on certain K3 surfaces in section \ref{k3}, the proofs of the theorems are given in section \ref{main}. We finish with some open questions in section \ref{oq}.

\section{Some curves on a K3-surface}\label{k3}

Let $g,d,s$ be integers with 
\begin{equation}  \label{eq2.1}
 d=g-s>0,\ g\ge0,\ g\ge 2s+13\mbox{ and }(d,g)\ne(7,4).
\end{equation}
Note that
$$d^2-12g=g(g-2s-12)+s^2>0.$$
It follows from \cite[Theorem 6.1,2]{k} that there exists a smooth K3-surface $S$ of type $(2,3)$ in $\PP^4$ containing a smooth curve $C$ of genus $g$ and degree $d$ with
$$
\Pic(S) = H \ZZ \oplus C \ZZ,
$$
where $H$ denotes the hyperplane bundle. In particular, we have
$$
H^2 =6, \ C \cdot H = d  \mbox{ and }  C^2 = 2g-2.
$$ 

\begin{prop} \label{prop2.1}
Suppose \eqref{eq2.1} holds, $g\ge2$ and $g+s>2$. Then the curve $C$ is an ample divisor on $S$.
\end{prop}

\begin{proof}
We show that $C \cdot D > 0$ for any effective divisor $D$ on $S$ which we may assume to be irreducible.
So let $D \sim mH + nC$ be an irreducible curve on $S$.
We have
$$
C \cdot D = m(g-s) + n(2g-2).
$$ 
Note first that, since $H$ is a hyperplane,
\begin{equation} \label{eq102}
D \cdot H = 6m + (g-s)n  > 0.
\end{equation}
If $m,n \geq 0$, then one of them has to be positive and then clearly $C \cdot D > 0$. The case $m,n \leq 0$ 
contradicts \eqref{eq102}.

\medskip
Suppose $m >0$ and $n < 0$.  Then, using \eqref{eq102} and \eqref{eq2.1}, we have
$$
C \cdot D = m(g-s) + n(2g-2) > - \frac{n}6\big(g(g-2s-12)+s^2+12\big)>0. 
$$ 

Finally, suppose $m< 0$ and $n > 0$. Then, since we assumed $D$ irreducible, we have $D^2\ge-2$ and
\begin{equation}\label{eq103}
nC \cdot D = -m D \cdot H + D^2 \geq -m D \cdot H -2 \geq -m -2.
\end{equation}
If $m \leq -3$ , then $n C \cdot D > 0$.  
If $m = -1$, we have
$$
C \cdot D = -(g-s) + n(2g-2) \geq g +s -2 > 0. 
$$
The same argument works for $m = -2,\; n \geq 2$. Finally, if $m = -2$ and $n = 1$, we have
$$D^2=(C-2H)^2=2g-2-4d+24=4s-2g+22.$$
So $D^2=-2$ if and only if $g=2s+12$, contradicting \eqref{eq2.1}. Thus $D^2\ge0$ and \eqref{eq103} implies that $C\cdot D>0$.
\end{proof}

We now investigate the possible existence of $(-2)$-curves on $S$. Note that, if $D$ is an irreducible effective divisor on $S$, we have 
$$
\chi(D) = \frac{D^2}{2} + 2 \geq 1,
$$
with equality if and only if $D$ is a $(-2)$-curve. It follows that a fixed component of any effective divisor must be a $(-2)$-curve. Note that any irreducible $(-2)$-curve $F$ has
$$h^0(S,F)=1,\  h^1(S,F)=h^2(S,F)=0$$
(see \cite{sd}).

\begin{prop}   \label{lem2.1}
Suppose that \eqref{eq2.1} holds and let $F$ be an irreducible $(-2)$-curve on $S$. Then 
one of the following holds:
\begin{itemize}
\item $F \cdot H \geq d-5$;
\item $s=-3$, $F\cdot H=d-6$, $F\sim C-H$;
\item $s=-3$, $g\equiv 0\bmod3$, $F\cdot H=d-6$, $F\sim\frac{g}3H-C$;
\item $s\ge-1$, $g=4s+16$, $F\cdot H=d-8$, $F\sim(s+4)H-C$;
\item $s\ge1$ and odd, $g=\frac52(s+5)$, $F\cdot H=d-10$, $F\sim\frac{s+5}2H-C$.
\end{itemize}
\end{prop}

\begin{proof}
Write $F \sim mH + nC$ and 
$$
r := F \cdot H = 6m +dn.
$$
The condition $F^2 = -2$ translates to
\begin{equation*} 
3m^2 + dmn + (g-1)n^2 = -1. 
\end{equation*}
Inserting $m = \frac{r-dn}{6}$, this gives
\begin{equation} \label{eq2.2}
 n^2[d^2 - 12 (g-1)] = r^2 + 12.
\end{equation}
Suppose first that $n^2 \geq 4$ and $r\le d-6$. In order to get a contradiction, it is enough to have
$$
4[d^2 - 12(g-1)] > (d-6)^2 + 12,
$$
which gives
$$
d^2 + 4d > 16g.
$$
Inserting $d = g-s$, this is equivalent to                                                               
$$
g(g-2s-12) +s^2 -4s > 0.
$$
This holds by \eqref{eq2.1}. 

\medskip
 It remains to consider the case $n^2 = 1$. If $r \leq d-12$, then in order to get a contradiction, it is enough to have
\begin{equation*}  
d^2 - 12(g-1) > (d-12)^2 + 12, 
\end{equation*}
which means $8d > 4g + 48$. Inserting $d = g-s$, this is equivalent to
$
g > 2s + 12
$,
which is valid by \eqref{eq2.1}. 

\medskip
The equation \eqref{eq2.2} with $n^2=1$ implies that $r-d$ is even. So we need to consider the cases $r=d-6$, $r=d-8$ and $r=d-10$.

\medskip
If $r = d-6$, \eqref{eq2.2} reduces to $d = g + 3$, so $s=-3$ and $m=\frac{d-6\pm d}6$, giving the second and third cases of the statement.

\medskip
Suppose $r = d-8$. Then \eqref{eq2.2} says 
$$
d^2 -12(g-1) = (d-8)^2 + 12,
$$
which reduces to $4d = 3g + 16$ or equivalently to
$g = 4s + 16$.
If $n = 1$, the formula $m = \frac{r-d}{6}$ gives $m = - \frac{4}{3}$, a contradiction. We are left with the case $n = -1$ and
$$
m = \frac{r+d}{6} = \frac{2d-8}{6} = s + 4.
$$
The condition $s\ge-1$ follows from \eqref{eq2.1}.

\medskip
Finally, if $r=d-10$, \eqref{eq2.2} reduces to $5d=3g+25$ or equivalently to $g=\frac52(s+5)$. So $m=\frac{d-10\pm d}6$. Again $n=1$ gives a contradiction, so $n=-1$ and 
$$m=\frac{2d-10}6=\frac{s+5}2.$$
The condition $s\ge1$ follows from \eqref{eq2.1}.
\end{proof}

\begin{cor} \label{cor2.2}
Suppose \eqref{eq2.1} holds with $s\ge-1$. Then the linear system $|C-H|$ is without fixed components. 
\end{cor}

\begin{proof}
Observe first that $|C -H|$ is effective and has $h^0(C-H) \geq 3$, since $(C-H)^2 = 2s+4 \geq 2$. Assume $|C - H|$ admits fixed components. 
Choose one of them and denote it by $F$. Note that $F$ is a $(-2)$-curve. So we may write
$$
C -H \sim M + F.
$$
Then
$$ 
2 < M \cdot H = (C -H) \cdot H - F \cdot H = d-6 - F \cdot H.
$$
So 
$$
F \cdot H \leq d-9.
$$
By Proposition \ref{lem2.1}, the only possibility is 
$$s\ge1,\ g=\frac52(s+5),\ F\sim\frac{s+5}2H-C.$$
In this case,
\begin{eqnarray*}
M\cdot C=\left(2C-\frac{s+7}2H\right).C&=&4g-4-\frac{s+7}2d\\
&=&10s+46-\frac{s+7}4(3s+25)\\
&=&-\frac14(3s^2+6s-9)\le0.
\end{eqnarray*}
This contradicts Proposition \ref{prop2.1}.
\end{proof}

\begin{cor}  \label{cor2.3}
 Suppose \eqref{eq2.1} holds with $s\ge-1$.
Let $D$ be an effective divisor on $S$ with $h^0(S,D) \geq 2$ and $h^0(S,C-D) \geq 2$. Then
the linear systems $|D|$ and $|C -D|$ have no fixed components.
\end{cor}

\begin{proof}
Since the statement is symmetric in $D$ and $C-D$, it is sufficient to prove the corollary for $C-D$. 

\medskip
Suppose $F$ is a $(-2)$-curve in the base locus of $|C -D|$. We may write
$$
C -D \sim M + F.
$$
Since $h^0(S,M)=h^0(S,C-D) \geq 2$, we have 
$$
3\le M\cdot H = (C-D) \cdot H -F \cdot H = d - D \cdot H - F \cdot H.
$$ 
Since $h^0(S,D) \geq 2$, we have  $D \cdot H \geq 3$. So
$$
1 \leq F \cdot H \leq d-6.
$$

By Proposition \ref{lem2.1}, the case $F\cdot H=d-6$ cannot occur since we are assuming $s\ge-1$ and we are left with the possibilities
\begin{equation}\label{eq104}
g=4s+16,\ F\cdot H=d-8,\ F\sim(s+4)H-C
\end{equation}
and
\begin{equation}\label{eq105}
g=\frac52(s+5),\ F\cdot H=d-10,\ F\sim\frac{s+5}2H-C.
\end{equation}
Moreover, since $|D|$ and $|C-D-F|$ are both effective, so is $|C-F|$. It follows from Proposition \ref{prop2.1} that $(C-F)\cdot C>0$.

\medskip
For \eqref{eq104}, we have
\begin{eqnarray*}
(C-F)\cdot C=\big(2C-(s+4)H\big)\cdot C&=&4g-4-(s+4)d\\
&=&16s+60-(s+4)(3s+16)\\
&=&-(3s^2+12s+4).
\end{eqnarray*}
This contradicts the fact that $(C-F)\cdot C>0$ except when $s=-1$.

\medskip
For \eqref{eq105}, we argue similarly. We have
\begin{eqnarray*}
(C-F)\cdot C=\left(2C-\frac{s+5}2H\right)\cdot C&=&2(5s+25)-4-\frac{s+5}4(3s+25)\\
&=&-\frac14(3s^2-59).
\end{eqnarray*}
Since $s$ is odd and $s\ge1$, this is a contradiction except for $s=1$ and $s=3$.

\medskip
This leaves us with the three possibilities
 \begin{equation}\label{eq101}
 (g,s)=(12,-1),\ (15,1),\ (20,3).
 \end{equation}
 In these cases, it is not sufficient to consider $(C-F)\cdot C$. However, in all three cases, we can show that the two conditions
 \begin{equation}\label{eq120}
 h^0(D)\ge2,\ h^0(C-D-F)\ge2
 \end{equation}
 lead to a contradiction. Note that \eqref{eq120} implies that $D\cdot H\ge3$ and $(C-D-F)\cdot H \ge3$ and hence 
 \begin{equation}\label{eq123}
 F\cdot H+3\le (C-D)\cdot H\le C\cdot H-3.
 \end{equation}
 Similarly, using Proposition \ref{prop2.1}, we obtain
  \begin{equation}\label{eq124}
 F\cdot C+1\le (C-D)\cdot C\le C\cdot C-1.
 \end{equation}
 
 Suppose first that $(g,s)=(12,-1)$, so that \eqref{eq2.1} and \eqref{eq104} give
 \begin{equation*}
 C\cdot H=13,\  F\cdot H=5,\ F\sim 3H-C,\ F\cdot C=17.
 \end{equation*}
 Writing $C-D\sim mH+nC$, \eqref{eq123} and \eqref{eq124} give
 \begin{equation}\label{eq125}
 8\le6m+13n\le10
 \end{equation}
 and
 \begin{equation}\label{eq126}
 18\le13m+22n\le21.
 \end{equation}
 Now $13\times\eqref{eq125}-6\times\eqref{eq126}$ gives
 $$-22\le37n\le22,$$so $n=0$. But now \eqref{eq125} gives an immediate contradiction.
 
 \medskip
 Next suppose that $(g,s)=(15,1)$. Then \eqref{eq2.1} and \eqref{eq105} give
 \begin{equation*}
 C\cdot H=14,\  F\cdot H=4,\ F\sim 3H-C,\ F\cdot C=14.
 \end{equation*}
 So \eqref{eq124} gives
  \begin{equation}\label{eq128}
 15\le14m+28n\le27.
 \end{equation}
 Since $14m+28n$ is divisible by $14$, this is an immediate contradiction.

\medskip
 The final case $(g,s)=(20,3)$ is a little more complicated (but also  more interesting). Here \eqref{eq2.1} and \eqref{eq105} give
 \begin{equation*}
 C\cdot H=17,\  F\cdot H=7,\ F\sim 4H-C,\ F\cdot C=30.
 \end{equation*}
 So \eqref{eq123} and \eqref{eq124} give
  \begin{equation}\label{eq121}
 10\le 6m+17n\le14
 \end{equation}
 and
\begin{equation}\label{eq122}
31\le17m+38n\le 37.
\end{equation}
Now $17\times\eqref{eq121}-6\times\eqref{eq122}$ gives 
$$-52\le61n\le52,$$
i.e. $n=0$. Now \eqref{eq121} gives $m=2$, which also satisfies \eqref{eq122}. Hence we must have $C-D\sim 2H$. But then $|C-D|$ does not have a fixed component. This is a contradiction.
\end{proof}

We now consider curves $D$ on $S$ with $D^2=0$.

\begin{prop}\label{prop2.5}
Suppose that \eqref{eq2.1} holds with $s\ge-1$ and let $D$ be an effective divisor with $D^2=0$ and without fixed components. Then $D\sim rE$ for some integer $r$, where $E$ is irreducible with $E^2=0$ and $D=E_1+\ldots +E_r$ with $E_i\sim E$. Moreover one of the following holds:
\begin{itemize}
\item $s\ge0$, $g=4s+13$, $E\sim(s+3)H-C$ or $E\sim 3C-4H$;
\item $s\ge4$ and even, $g=\frac{5s}2+11, E\sim\frac{s+4}2H-C$ or $E\sim3C-5H$.
\end{itemize}
\end{prop}
\begin{proof}
By a result in \cite{sd} (see \cite[Proposition 2.1]{f} for a statement), $D=E_1+\ldots E_r\sim rE$ as in the statement. We need only check that $E$ has one of the stated forms. For this, let $E\sim mH+nC$, so that
\begin{equation}\label{eq106}
E^2 = 6m^2 + 2dmn + (2g-2)n^2.
\end{equation}
For an integer solution of the equation $E^2=0$, we require the discriminant $d^2 - 6(2g-2)$ of \eqref{eq106} to be a perfect square. So suppose 
$$d^2 -6(2g-2) = g^2 - (2s + 12)g + s^2 +12 = t^2$$
for some $t\ge0$, i.e.
\begin{equation}\label{eq107}
(g-s-6)^2-t^2=12s+24.
\end{equation}
Write $g-s-6=t+2b$. Since $s\ge-1$, \eqref{eq2.1} implies that $b>0$ and  $t\ge\max\{s+7-2b,0\}$. The equation \eqref{eq107} gives
\begin{equation}\label{eq110}
b(t+2b)=3s+6+b^2,
\end{equation}
so that 
\begin{equation}\label{eq108}
b^2\ge b(s+7)-3s-6
\end{equation}
On the other hand, since $bt\ge0$, $b^2=3s+6-bt\le3s+6$. Combining this with \eqref{eq108}, we get
\begin{equation}\label{eq109}
b(s+7)\le6s+12.
\end{equation}
If $b\ge6$, \eqref{eq109} gives an immediate contradiction. For $3\le b\le5$, we can calculate $t$ directly from \eqref{eq110} and show that $t+2b<s+7$. This leaves us with $b=1$ and $b=2$.

\medskip
When $b=1$, \eqref{eq110} gives $t=3s+5$ and $g=t+2b+s+6=4s+13$. The equation $E^2=0$ (see \eqref{eq106}) now gives
$$\frac{m}n=\frac{-d\pm t}6=-\frac43\mbox{ or }-(s+3).$$
When $b=2$, we get similarly $t=\frac{3s+2}2$, $g=\frac{5s}2+11$ and
$$\frac{m}n=-\frac53\mbox{ or } -\frac{s+4}2.$$
The restrictions on $s$ come from \eqref{eq2.1}. To see in each case that there is an effective divisor $E$ in the given divisor class, one checks that $E\cdot H>0$. Since $E$ is primitive, it must also be irreducible.
\end{proof}

\begin{cor}\label{cor2.4}
Suppose that \eqref{eq2.1} holds with $s\ge-1$ and that $D$ and $C-D$ are effective divisors without fixed components. Then
\begin{itemize}
\item[(i)] $D^2\ne0$, $(C-D)^2\ne0$;
\item[(ii)] $h^0(C,D|_C)=h^0(S,D)=\frac{D^2}2+2$.
\end{itemize}
\end{cor}
\begin{proof} (i) Suppose that $(C-D)^2=0$. By the proposition, we have $C-D=rE$ with $E$ as in the statement. Moreover $r\ge1$ since $C-D$ is effective and $E\cdot C\ge0$ (in fact $E\cdot C>0$ in view of Proposition \ref{prop2.1}). Since also $E^2=0$, we have  
$$D^2=C^2-2rE\cdot C=C\cdot (C-2rE)\le C\cdot (C-2E).$$
Using the values of $E$ from the proposition, we see that $D^2<0$, contradicting the assumption that $D$ has no fixed components. Interchanging $D$ and $C-D$ in this argument, we obtain a similar contradiction when $D^2=0$.

(ii) By (i), $(C-D)^2>0$, so the results of \cite{sd} (\cite[Proposition 2.1]{f}) apply to show that the general member of $|C-D|$ is smooth and irreducible and 
$$h^1(S,D-C)=h^1(S,C-D)=0.$$ 
Moreover, $D-C$ is not effective, so $h^0(S,D-C)=0$. The first equality in (ii) now follows from the cohomology sequence
$$
0 \ra H^0(S,D-C) \ra H^0(S,D) \ra H^0(C,D|_C) \ra H^1(S,D-C).
$$
For the second equality, we note that (i) implies that $h^1(S,D)=0$ and $h^2(S,D)=h^0(S,-D)=0$, so
$$h^0(S,D)=\chi(D)=\frac{D^2}2+2.$$
\end{proof}

\section{Proof of theorems}\label{main}

In this section we prove our main theorems. We start with a lemma.

\begin{lem} \label{lem3.1}
Suppose that \eqref{eq2.1} holds with $s \geq -1$.
Then $H|_C$ is a generated line bundle on $C$ with $h^0(C,H|_C) = 5$ and 
$$
S^2H^0(C,H|_C) \ra H^0(C,H^2|_C)
$$ 
is not injective.
\end{lem}

\begin{proof}
Consider the exact sequence
$$
0 \ra \cO_S(H-C) \ra \cO_S(H) \ra \cO_C(H|_C) \ra 0.
$$
$H -C$ is not effective, since $(H-C) \cdot H = 6 -d < 0$. So we have 
$$
0 \ra H^0(S,H) \ra H^0(C,H|_C) \ra H^1(S,H-C) \ra 0.
$$
Now 
$$
(C-H)^2 = 2g-2 -2d + 6 = 2 s + 4 \geq 2,
$$
from which it follows that $|C-H|$ is effective. Since $|C-H|$ has no fixed component by Corollary \ref{cor2.2}, it follows that its general element is smooth and irreducible (see \cite{sd} or \cite[Proposition 2.1]{f}). Hence $h^1(S,H-C) = 0$ and 
therefore
$h^0(C,H|_C) = h^0(S,H) =5$. The last assertion follows from the fact that $S$ is contained in a quadric. 
\end{proof}

\begin{cor}\label{cor3.2}
Suppose that \eqref{eq2.1} holds with $s\ge-1$ and $\Cl(C)=\left[\frac{g-1}2\right]$. Then there exists a stable vector bundle of rank $2$ and degree $g-s$ on $C$ with $h^0(E)=4$.
\end{cor}
\begin{proof}
Note that $g-s<2(\Cl(C)+2)$. The result now follows from the lemma and \cite[Lemma 3.3]{cl}.
\end{proof}

\begin{theorem}    \label{thm3.1}
Suppose that   $g$, $s$, $d$ are integers such that
\begin{equation}\label{eq11}
s \geq -1,\ g \geq 2s+14\mbox{ and } d=g-s.
\end{equation}
Then there exists a curve $C$ of genus $g$ having $\Cl(C) = \left[ \frac{g-1}{2} \right]$ 
and a stable vector bundle $E$ of rank $2$ and degree $d$ on $C$ with $\gamma(E) = \frac{g-s}{2} - 2$.  Hence
$$
\Cl_2(C) \le \frac{g-s}{2} -2 <\Cl(C).
$$
\end{theorem}

\begin{proof}
Let $S$ and $C$ be as at the beginning of section \ref{k3}. In view of Corollary \ref{cor3.2}, it is sufficient to prove that $\Cl(C)=\left[\frac{g-1}2\right]$. Since $C$ is ample by Proposition \ref{prop2.1}, it follows from \cite[Proposition 3.3]{cp} that $\Cl(C)$ is 
computed by a pencil. If $\Cl(C) < \left[\frac{g-1}{2} \right]$, it then follows from \cite{dm} (see also \cite[Proposition 3.1]{f}) that
there is an effective divisor $D$ on $S$ 
such that $D|_C$ computes $\Cl(C)$ and satisfying
\begin{equation}\label{eq30}
h^0(S,D) \geq 2, \quad h^0(S,C-D) \geq 2 \quad  \mbox{and} \quad \deg(D|_C) \leq g-1.
\end{equation}
By Corollaries \ref{cor2.3} and \ref{cor2.4}, we have$$
\Cl(C) = \Cl(D|_C) =  D \cdot C - D^2 -2.
$$
To obtain a contradiction, it is therefore sufficient to prove that
$$
D \cdot C - D^2 - 2 \geq \left[ \frac{g-1}{2} \right].
$$
Writing $D \sim mH + nC$ with $m,n \in \ZZ$, we have $D \cdot C - D^2 - 2=f(m,n)$, where
\begin{equation*}
f(m,n) := -6m^2 + (1 -2n)dm + (n-n^2)(2g-2) -2.
\end{equation*}
We therefore require to prove that
\begin{equation}\label{fmn}
f(m,n)\ge\left[\frac{g-1}2\right].
\end{equation}
 
By Corollaries \ref{cor2.3} and \ref{cor2.4}, 
we have $D^2 > 0$. Also, by \eqref{eq30},  $D\cdot H\ge3$ and $(C-D)\cdot H\ge3$, hence $D\cdot H\le d-3$.  These inequalities and $\deg(D|_C) \leq g-1$ translate to\begin{equation} \label{eq3.1}
3m^2 + mnd + n^2(g-1) > 0,
\end{equation}  
\begin{equation} \label{eq3.2}
3\le 6m +nd \le d-3,
\end{equation}
\begin{equation} \label{eq3.3}
md + (2n-1)(g-1) \leq 0.
\end{equation}
We shall prove that \eqref{eq3.1} -- \eqref{eq3.3} imply \eqref{fmn}.

\medskip
Denote by
$$
a := \frac{1}{6}(d + \sqrt{d^2 -12(g-1)}) \quad \mbox{and}  \quad b:= \frac{1}{6}(d - \sqrt{d^2 -12(g-1)})
$$
the solutions of the equation $6x^2 - 2dx + 2g-2 = 0$. 
Note that $d^2 > 12(g-1)$. So $a$ and $b$ are positive real numbers; moreover, substituting $g=d+s$, we see that, since $s\ge-1$ and $d\ge s+14$, 
$$(d-12)^2<d^2-12(g-1)<(d-6)^2.$$
Hence
\begin{equation} \label{eq3.5}
1 < b < 2.
\end{equation}
Moreover, if $n\ne0$, \eqref{eq3.1} holds if and only if
\begin{equation}\label{m/n}
\frac{m}n<-a \mbox{ or }\frac{m}n>-b.
\end{equation}

If $n < 0$ and $\frac{m}n>-b$, then \eqref{eq3.2}
implies that $3 < n(d-6b) < 0$, because $n < 0$ and $d -6b = \sqrt{d^2 - 12(g-1)} > 0$, which gives a contradiction. Similarly, if $n>0$ and $\frac{m}n<-a$, we obtain $3<n(d-6a)<0$, again a contradiction. In view of \eqref{m/n}, it remains to consider the three possibilities
\begin{itemize}
\item $n<0$, $m>-an$;
\item $n>0$, $m>-bn$;
\item $n=0$.
\end{itemize}
In each case, we use \eqref{eq3.3} to prove \eqref{fmn}.

\medskip
If $n < 0$ and $m > -an$, we get from \eqref{eq3.3} 
$$
-an < m \leq \frac{(g-1)(1-2n)}{d} < \frac{(1-2n)d}{12},
$$
since $d^2 > 12(g-1)$. For a fixed $n$, $f(m,n)$ is strictly increasing as a function of $m$ for $m \leq \frac{(1-2n)d}{12}$
and therefore
\begin{eqnarray*}
f(m,n) &>& f(-an,n)\\
& = & \frac{d^2 - 12(g-1) +d\sqrt{d^2 - 12(g-1)}}{6} \cdot(-n) -2 \\
& \geq & \frac{d^2 - 12(g-1) +d\sqrt{d^2 - 12(g-1)}}{6}  -2. 
\end{eqnarray*}
The inequality \eqref{fmn} therefore holds if
$$
d^2-15g+3+d \sqrt{d^2 -12(g-1)} \geq 0.
$$ 
Since $d^2>12g$, it is therefore sufficient to prove that
$$d^2-15g+3d+3\ge0,$$ or equivalently $g(g-2s-12)+s^2-3s+3\ge0$. This is certainly true under our hypotheses. 

\medskip
If $n>0$ and $m > -bn$, \eqref{eq3.3} and \eqref{eq3.5} give 
\begin{equation} \label{eq3.4}
-(2n-1) \le m \leq -\frac{(g-1)(2n-1)}{d}. 
\end{equation}
For a fixed $n\ge1$, $f(m,n)$ is strictly decreasing for
$m\ge -\frac{(2n-1)d}{12}$ and hence throughout the range \eqref{eq3.4} (whenever this range is non-empty).
So 
\begin{eqnarray*}
f(m,n)-\frac{g-1}2& \geq&
f \left( -\frac{(g-1)(2n-1)}{d},n \right)-\frac{g-1}2\\ 
&=& \frac{g-1}{2} (2n-1)^2 \left( 1 - \frac{12(g-1)}{d^2} \right) -2\\
&\ge&\frac{g-1}2(2n-1)^2\left(1-\frac{g-1}g\right)-2\\
&=&\frac{g-1}{2g}(2n-1)^2-2\ge0 \mbox{ for }n\ge2.
\end{eqnarray*}
If $n=1$, then \eqref{eq3.4} gives $m=-1$ and
\begin{equation}\label{eqnew}
f(-1,1)=d-8\ge\left[\frac{g-1}2\right]\mbox{ for }g\ge2s+14.
\end{equation}

Finally, suppose $n=0$. Then
$$
f(m,0) = -6m^2 + dm -2.
$$
As a function of $m$ this takes its maximum value at $\frac{d}{12}$. By \eqref{eq3.1} and \eqref{eq3.3}, 
$$1\le m \leq \frac{g-1}{d} \leq \frac{d}{12}.$$ 
So $f(m,0)$ takes its minimal value in the allowable range at $m=1$. Hence 
\begin{equation}\label{eqnew2}
f(m,0)\ge f(1,0) = d-8\ge\left[\frac{g-1}2\right]\mbox{ for }g\ge2s+14.
\end{equation} 

\end{proof}

\begin{rem}\label{rmk3.3}
\begin{em}The case $s=-1$, $g$ even, is \cite[Theorem 3.7]{fo}. The case $s=-2$, $g$ odd (not included in our theorem) is \cite[Theorem 1.4]{fo}.
\end{em}\end{rem}

\begin{rem}\label{rmk3.4}
\begin{em}The result of Theorem \ref{thm3.1} is best possible in the sense that it fails for $g=2s+13$. In this case
$$\gamma(E)=\frac{g-s}2-2<\frac12\left[\frac{g-1}2\right]+2,$$
which contradicts \cite[Proposition 3.8]{cliff} if $\Cl(C)=\left[\frac{g-1}2\right]$. The points of failure in the proof are when $(m,n)=(1,0)$ and $(m,n)=(-1,1)$ (see \eqref{eqnew} and \eqref{eqnew2}), i.e. for $D\sim H$ and $D\sim C-H$. In fact $H|_C$ contributes to $\Cl(C)$, so, when $g=2s+13$,
$$\Cl(C)\le d-8<\left[\frac{g-1}2\right].$$
When $g=2s+14$ or $g=2s+15$, we have 
$d=\Cl(C)+8$, so that $H|_C$ computes the Clifford index. Thus $\Cl(C)$ is realised by an embedding of $C$ in ${\mathbb P}^4$, although the Clifford dimension of $C$ is $1$, i.e. $\Cl(C)$ is computed by a pencil (a fact used in the proof of Theorem \ref{thm3.1}).
\end{em}\end{rem}

\begin{cor}\label{cor11}
For $g \geq 12$, there exists a curve $C$ of maximal Clifford index $\left[ \frac{g-1}{2} \right]$ such that
$$
\Cl_2(C)=\frac{1}{2} \left[ \frac{g-1}{2} \right] +2.$$\end{cor}

\begin{proof}
Taking $s=\left[\frac{g-14}2\right]$ in the theorem, we obtain
$$\Cl_2(C) \le \frac{g-s}{2} -2=\frac12\left[\frac{g-1}2\right]+2=\frac12\Cl(C)+2.$$
For the opposite inequality, see \cite[Proposition 3.8]{cliff}.
\end{proof}

\begin{rem}\label{rmk3.7}\begin{em}
The result also holds for $g=11$ \cite[Theorem 1.4]{fo}. For $g\le10$, we have $\Cl(C)\le4$ for all $C$ and $\Cl_2(C)=\Cl(C)$ by \cite[Proposition 3.8]{cliff}.
\end{em}\end{rem}

Finally, we can express Corollary \ref{cor11} in terms of $\Cl(C)$ rather than $g$. Although this is technically a corollary of Theorem \ref{thm3.1}, it is of sufficient interest for us to state it as a theorem.

\begin{theorem}\label{thm2}
Let $\gamma$ be an integer, $\gamma\ge5$. Then there exists a curve $C$ with $\Cl(C)=\gamma$ such that
$$\Cl_2(C)=\frac{\gamma}2+2.$$
Moreover $C$ can be taken to have genus either $2\gamma+1$ or $2\gamma+2$.
\end{theorem}
\begin{proof}
For $\gamma\ge6$, this is a restatement of Corollary \ref{cor11}. For $\gamma=5$, we need also Remark \ref{rmk3.7}.
\end{proof}

\section{Open Questions}\label{oq}
The following question (Mercat's conjecture for rank $2$ and general $C$ -- see \cite{m} and \cite[Proposition 2.7]{ln}) remains open.

\begin{ques}\label{q1} Is it true that $\Cl_2(C)=\Cl(C)$ for the general curve $C$ of any genus?
\end{ques}

Farkas and Ortega conjectured in \cite{fo} that the answer to this question is yes and proved this for $g\le19$ (for a proof when $g\le16$, see \cite[Theorem 1.7]{fo}). If the answer is yes, we can ask a more precise question, the answer to which is known only for $g\le10$ (or equivalently $\Cl(C)\le4$ (see \cite[Proposition 3.8]{cliff})).

\begin{ques}\label{q2} Is it true that $\Cl_2(C)=\Cl(C)$ whenever $C$ is a Petri curve?
\end{ques}

It may be noted that none of the curves constructed in this paper or in \cite{fo,fo2,cl,ln} is general (they all lie on K3 surfaces with Picard number $2$). Some of the curves are definitely not Petri (in particular those of Corollary \ref{cor11} and Theorem \ref{thm2}); however it remains possible that some are Petri.

Note also that, for any $\gamma$, there exist curves with 
$$\Cl_2(C)=\Cl(C)=\gamma$$ 
(for example, smooth plane curves of degree $\gamma+4$ -- see \cite[Proposition 8.1]{cliff}).

\begin{ques}\label{q3}
Suppose $\frac{\gamma}2+2<\gamma'<\gamma$. Does there exist a curve $C$ with $\Cl(C)=\gamma$ and $\Cl_2(C)=\gamma'$?
\end{ques}

\end{document}